\documentclass{amsart}
\usepackage{amsmath,amssymb,amscd,graphicx,psfrag,epsf,amsthm
}
\usepackage[all]{xy}

\newcommand{\A}{{\mathbb {A}}}
\newcommand{\Q}{{\mathbb Q}}

\newcommand{\cC}{{\mathcal {C}}}

\newcommand{\cM}{{\mathcal {M}}}
\newcommand{\cF}{{\mathcal {F}}}

\newcommand{\cE}{{\mathcal {E}}}

\newcommand{\cL}{{\mathcal {L}}}

\newtheorem{thm}{Theorem}%[section]

\newtheorem{lemma}[thm]{Lemma}
\newtheorem{cor}[thm]{Corollary}

\newtheorem{prop}[thm]{Proposition}

\newtheorem{conj}[thm]{Conjecture}
\newtheorem*{thm*}{Theorem}
\newtheorem*{thmF1}{Theorem F1}
\newtheorem*{thmF2}{Theorem F2}

\theoremstyle{definition}

\newtheorem{criterion}[thm]{Criterion}

\newtheorem{exmp}[thm]{Example}

\newtheorem{notation}[thm]{Notation-Remark}

\newtheorem{defn-thm}[thm]{Definition-Theorem}  %!!!!!!!!!!!!!!!!!!!!!!!!
\newtheorem{rem}[thm]{Remark}

\theoremstyle{remark}

\renewcommand{\c}[0]{{\mathbb C}}

\renewcommand{\o}[0]{{\mathcal O}}
\newcommand{\z}[0]{{\mathbb Z}}

\newcommand{\p}[0]{{\mathbb P}}

\newcommand{\pic}[0]{\operatorname{Pic}}

\newcommand{\codim}[0]{\operatorname{codim}}

\newcommand{\Proj}[0]{\operatorname{Proj}}

\newcommand{\sym}[0]{\operatorname{Sym}}

\def\into{\DOTSB\lhook\joinrel\rightarrow}

\numberwithin{equation}{section}

\author{Carolina Araujo}
\author{Jos\'e J. Ram\'on-Mar\'i}

\address{Carolina Araujo: \sf IMPA, Estrada Dona Castorina 110, Rio de
  Janeiro, RJ 22460-320, Brazil}
\email{caraujo@impa.br}

\address{Jos\'e J. Ram\'on-Mar\'i: \sf IMPA, Estrada Dona Castorina 110, Rio de
  Janeiro, RJ 22460-320, Brazil}
\email{josemari@impa.br}

\title{Flat deformations of $\p^n$}
\begin{document}

\maketitle

\begin{abstract}
In this paper we study projective flat deformations of $\p^n$.
We prove that the singular fibers of a projective flat deformation of $\p^n$ appear either
in codimension $1$ or over singular points of the base.
We also describe projective flat deformations of $\p^n$ with smooth total space,
and discuss flatness criteria.
\end{abstract}

%%%%%%%%%%%%%%%%%%%%%%%%%%%%%%%%%%%%%%%%%%%%%%%%%%%%%%%%%
%                                                       %
%                     1)  INTRODUCTION                  %
%                                                       %
%%%%%%%%%%%%%%%%%%%%%%%%%%%%%%%%%%%%%%%%%%%%%%%%%%%%%%%%%

%\section*{}

\section{Introduction}

It is a well-known theorem of Siu that $\p^n$ is rigid  (see \cite[Main Theorem]{siu_rigidity}).
This means that, if $\pi:X\to Y$ is a smooth proper morphism between
connected complex manifolds, and if the general fiber of $\pi$ is isomorphic to $\p^n$, then
every fiber of $\pi$ is isomorphic to $\p^n$.
The aim of these notes is to prove similar results for projective \emph{flat} deformations of $\p^n$.

A $\p^n$-bundle is  a smooth projective morphism between complex analytic spaces whose
fibers are all isomorphic to  $\p^n$.
The simplest examples  are \emph{scrolls}.
These are  $\p^n$-bundles $\pi:X\to Y$ satisfying the following equivalent conditions.
\begin{enumerate}
	\item There is a locally free sheaf $\cE$ of rank $n+1$ on  $Y$, and an isomorphism $X\cong \p(\cE)$
	over $Y$.
	\item There is a line bundle $\cL$ on $X$ whose restriction to every fiber $X_t$
	satisfies $\cL|_{X_t}\cong \o_{\p^n}(1)$.
	\item The morphism $\pi$ admits a section $\sigma:Y\to X$.
\end{enumerate}
We call a line bundle $\cL$ as in (2)  a  \emph{global $\o(1)$ for  $\pi$}.
Given $\cE$ as in (1), the
tautological line bundle $\o_{\p(\cE)}(1)$ is a global $\o(1)$ for  $\pi$.
Conversely, given $\cL$ as in (2), $\cE$ can be taken to be $\pi_*\cL$.
The equivalence with (3) can be seen by considering the associated
$\p^n$-bundle of hyperplanes (see \cite[p.134]{milne}).

We recall the following characterization of scrolls, due to  Fujita.

\begin{thmF1}[{\cite[Corollary 5.4]{fujita75}}]
Let $X$ and $Y$ be irreducible and reduced complex analytic spaces, and $\pi:X\to Y$ a proper flat morphism
whose fibers are all irreducible and reduced.
Suppose that the general fiber of $\pi$ is isomorphic to $\p^n$, and that
there exists a $\pi$-ample line bundle $\cL$ on $X$ such that
$\cL|_{X_t}\cong \o_{\p^n}(1)$ for general $t\in Y$.
Then $\pi$ is a $\p^n$-bundle and $\cL$ is a global $\o(1)$ for $\pi$.
\end{thmF1}

Every $\p^n$-bundle over a smooth curve carries a global $\o(1)$.
In general, not every $\p^n$-bundle is  a scroll,
although this is the case  locally in the \'etale topology.
(See \cite{artin} for the connection between this condition and Brauer-Severi varieties.)
So, it is natural to look for more general characterizations of $\p^n$-bundles,
without requiring the existence of a  global $\o(1)$.
We start by observing that Theorem F1 does not hold
if we drop the assumptions on the line bundle $\cL$.
This is illustrated in the following example.

\begin{exmp} \label{flat_degeneration_to_cone}
Let $\nu:\p^n\to \p^N$ be the $d$-uple embedding of $\p^n$, with $n,d\geq 2$.
Denote by $V$ the image of $\nu$ in $\p^N$, and let $\cC(V)\subset \p^{N+1}$ be the cone
over $V$ with vertex $P$.
Let $\Gamma$ be a general pencil of hyperplane sections of $\cC(V)$ in $\p^{N+1}$.
It gives rise to a projective flat morphism $\pi:X\to Y=\p^1$ whose fibers are precisely the members of $\Gamma$.
There is a unique member of $\Gamma$ that passes through $P$.
It is isomorphic to the cone over a general hyperplane sections of $V$ in $\p^{N}$.
Let $o\in Y$ be the point corresponding to this singular member of $\Gamma$.
Then $X_t\cong \p^n$ for every $t\in Y\setminus\{o\}$, while
$X_o$ is a singular cone.

We call the reader's attention to the following properties of $\pi$.
\begin{itemize}
	\item The locus of $Y$ over which the fibers are not isomorphic to $\p^n$ has codimension one in $Y$.
	\item The total space $X$ is singular.
\end{itemize}
We will see that these properties are typical for non-smooth flat deformations of $\p^n$.
\end{exmp}

The following is the key result in our study of  flat deformations of $\p^n$.
We henceforth denote the unit ball of $\c^m$ by $\Delta^m$.

\begin{thm} \label{flat_morphism_smooth_base}
Let $X$ be a complex analytic space, and $\pi:X\to \Delta^m$ a projective surjective flat morphism,
with $m\geq 2$. Suppose that $X_t\cong \p^n$ for every $t\in\Delta^m\setminus \{\bar 0\}$.
Then $X$ is smooth and $\pi$ is a scroll.
\end{thm}

As a consequence of Theorem~\ref{flat_morphism_smooth_base},
the singular fibers of a projective flat deformation of $\p^n$ appear either
in codimension $1$, or over rather singular points of the base.
To state this precisely, we introduce some notation.
Given a surjective morphism $\pi:X\rightarrow Y$ between  algebraic varieties,
we denote by $S_{\pi}$ the locus of points of $Y$ over which $\pi$ is not smooth.
It is a closed subset of $Y$.

\begin{cor}\label{cor}
Let $\pi:X\rightarrow Y$ be a projective surjective flat morphism between algebraic varieties
with general fiber isomorphic to $\p^n$, and fix $y\in S_{\pi}$.
Suppose that there is a surjective quasi-finite morphism from a smooth variety onto a neighborhood of
$y$ in $Y$.
Then $S_{\pi}$ has pure codimension $1$ at $y$.
\end{cor}

Next we describe projective flat deformations of $\p^n$ with smooth total space.
In order to state our result, we introduce some more notation.

\begin{notation}
Let $\pi:X\rightarrow Y$ be a proper surjective equidimensional morphism between normal algebraic varieties.
We denote by $R_{\pi}$ the locus of points of $Y$ over which the fibers of $\pi$ are reducible.
Note that $R_{\pi}$ is a constructible set.
Indeed, let $d$ denote the relative dimension of $\pi$.
Then, for every $y\in Y$, 
$H^{2d}(X_y,\z)$ is free and its rank is the number of irreducible components of $X_y$.
On the other hand,
the sheaf $R^{2d}\pi_*\z_X$ is constructible by the proper base change theorem on \'etale cohomology
and Artin's comparison theorem (see \cite[Theorem VI.2.1]{milne}).
\end{notation}

\begin{thm} \label{X_smooth_pi_flat}
Let $X$ be a smooth complex quasi-projective variety, $Y$ a normal complex quasi-projective variety,
and $\pi:X\to Y$ a proper surjective flat morphism.
Suppose that the general fiber of $\pi$ is isomorphic to $\p^n$. Then either
\begin{enumerate}
   \item  $Y$ is smooth and  $\pi:X\to Y$ is a $\p^n$-bundle; or
   \item $S_{\pi}$ is of pure codimension $1$, and $S_{\pi}=\overline{R_{\pi}}$.
\end{enumerate}
\end{thm}

The second case described in Theorem~\ref{X_smooth_pi_flat}
is exemplified by suitable blow-ups of $\p^n$-bundles.

It is also useful to have characterizations of $\p^n$-bundles without flatness assumptions.
These can be obtained by applying our flatness criteria discussed in Section~\ref{section:flatness}.

\bigskip

\noindent {\bf Notation.}
Our ground field is always $\c$.

Let $\pi:X\to Y$ be a morphism of complex analytic spaces.
Given $t\in Y$, we denote by $X_t$ the scheme-theoretical fiber over $t$.
We refer to the reduced scheme $(X_t)_{red}$ as the set-theoretical fiber over $t$.

Given a locally free sheaf $\cE$ on a complex analytic space $X$, we denote by
$\p(\cE)$  the Grothendieck projectivization $\Proj_Y(\sym (\cE))$.

Varieties are always assumed to be irreducible and reduced. 

%%%%%%%%%%%%%%%%%%%%%%%%%%%%%%%%%%%%%%%%%%%%%%%%%%%%%%%%%
%                                                       %
%                     SECTION 2                         %
%                                                       %
%%%%%%%%%%%%%%%%%%%%%%%%%%%%%%%%%%%%%%%%%%%%%%%%%%%%%%%%%

\section{Proofs}

In order to characterize $\p^n$-bundles that are not necessarily scrolls,
we will use the following lemma to construct ``local $\o(1)$'s''.
This result follows from the proof of  \cite[Lemma 3.3]{andreatta_mella97}.
We reprove it here for the reader's convenience.

\begin{lemma} \label{constructing_O(1)}
Let $X$ and $U$ be complex analytic spaces, with
$U\setminus\{o\}\cong \Delta^m\setminus \{\bar 0\}$ for some point $o\in U$ and some $m\geq 2$.
Let $\pi:X\to U$ be a surjective projective morphism whose restriction to
$X\setminus \pi^{-1}(o)$ is a $\p^n$-bundle, and assume that $\codim_X\big(\pi^{-1}(o)\big)\geq 2$.
Let $\cM$ be a $\pi$-ample line bundle on $X$, and let $d\in \z_{>0}$ be such that
$\cM|_{X_t}\cong \o_{\p^n}(d)$ for every $t\in U\setminus\{o\}$.
Then there exists a coherent sheaf $\cL$ on $X$ such that $\cL|_{X\setminus \pi^{-1}(o)}$
is invertible and $\big(\cL|_{X\setminus \pi^{-1}(o)}\big)^{\otimes d}\cong\cM|_{X\setminus \pi^{-1}(o)}$.
If, moreover, $X$ is smooth, then there is a line bundle $\cL$ on $X$ such that $\cL^{\otimes d}\cong \cM$.
\end{lemma}

\begin{proof}
Set $U^*=U\setminus\{o\}$, $X^*=X\setminus \pi^{-1}(o)$, and $\cM^*=\cM|_{X^*}$.
We will apply Leray spectral sequence to the morphism $\pi|_{X^*}:X^*\to U^*$ and the
locally constant sheaf $\z_{X^*}$.
Set $E_2^{p,q}:=H^p\big(U^*, R^q(\pi|_{X^*})_*\z_{X^*}\big)$, and denote the differentials of the corresponding spectral sequence by $d_r^{p,q}$.
The cohomology classes $c_1({\mathcal M})^k$ yield the Leray-Hirsch theorem for $\Q_{X^*}$ (see \cite[Theorem 7.33]{voisin1}
and its proof, or \cite[p.192-3]{bott_tu}).
Hence, $d_r^{p,q}\otimes \Q=0$ for $r\geq 2$, and abutment at $E_2$ follows by recursion on $r$, using that the $E_2^{p,q}$'s are free abelian groups.
So there is an isomorphism of $H^*(U^*,\z)$-algebras
$$
H^*(X^*,\z) \cong H^*(U^*,\z) \otimes H^*(\p^n,\z).
$$

In particular, $H^2(X^*,\z)\cong \z$, and the cokernel of the composed Chern class map
$$
c_1: \  \pic(X^*) \ \to \ H^2(X^*,\z)\ \cong \ \z
$$
is finite. On the other hand, this cokernel injects into $H^2(X^*,{\mathcal O}_{X^*})$, which is torsion-free.
Hence $c_1$ is surjective.

Note that $c_1(\cM^*)=d$. Since the kernel of $c_1$ is divisible, there is a line bundle $\cL^*$
on $X^*$ such that $(\cL^*)^{\otimes d}\cong\cM^*$.
We take $\cL$ to be a coherent sheaf on $X$ extending $\cL^*$
(see \cite[Theorem 1]{SerreNorm}).
If $X$ is smooth, then, since $\codim_X(X\setminus X^*)\geq 2$, there is an isomorphism $\pic(X^*)\cong \pic(X)$.
Let $\cL\in \pic(X)$ correspond to $\cL^*\in \pic(X^*)$.
Then   $\cL^{\otimes d}\cong \cM$.
\end{proof}

\begin{proof}[Proof of Theorem~\ref{flat_morphism_smooth_base}]
In order to show that $\pi$ is a $\p^n$-bundle,
it suffices to prove that the scheme-theoretical fiber $X_{\bar 0}$ over $\bar 0\in \Delta^m$ is isomorphic to $\p^n$.
By intersecting $\Delta^m$ with a $2$-plane passing through $\bar 0$, we may assume that $m=2$.

Let $\cM$ be a $\pi$-ample line bundle on $X$, and
let $d\in \z$ be such that $\cM|_{X_t}\cong \o_{\p^n}(d)$ for every $t\in\Delta^2 \setminus \{\bar 0\}$.
The restriction $\cM|_{X_{\bar 0}}$ is ample.
Hence, by replacing $\cM$ with a sufficiently high tensor power  if necessary, we may assume that $\cM|_{X_t}$ is very ample and that $H^i(X_t,\cM|_{X_t})=0$
for every $t\in \Delta^2$ and $i>0$.
Since $\pi$ is flat, $\chi(X_t,\cM|_{X_t})$ is constant on $t\in \Delta^2$,
and hence so is $h^0(X_t,\cM|_{X_t})$.
Set $\cF= \pi_*\cM$. Then $\cF$ is locally free and
the natural map $\pi^*\cF\to \cM$ is surjective
(see for example \cite[Theorem 1.4]{ueno_LNM439}).
This yields an embedding $i:X\into \p(\cF)$ over $\Delta^2$.

By Lemma~\ref{constructing_O(1)}, there is a coherent sheaf $\cL$ on $X$ such that
$\cL|_{X\setminus \pi^{-1}(\bar 0)}$ is invertible and
$\big(\cL|_{X\setminus \pi^{-1}(\bar 0)}\big)^{\otimes d}\cong\cM|_{X\setminus \pi^{-1}(\bar 0)}$.
Set $\cE=(\pi_*\cL)^{\vee\vee}$. Since $\Delta^2$ is smooth and two-dimensional,
$\cE$ is locally free (see for example \cite[II. Lemma 1.1.10]{OSS_vector_bdles_on_Pn}).
Consider the projectivization $p:\p(\cE)\to \Delta^2$, with tautological line bundle $\o_{\p(\cE)}(1)$.
There is an isomorphism $\varphi: \p(\cE)\setminus p^{-1}(\bar 0)\xrightarrow{\cong} X\setminus  \pi^{-1}(\bar 0)$
such that $\varphi^*\big(\cM|_{X\setminus  \pi^{-1}(\bar 0)}\big)=\o_{\p(\cE)}(d)\big|_{\p(\cE)\setminus p^{-1}(\bar 0)}$.

Note that $p_*\o_{\p(\cE)}(d)$ is locally free and that the natural map $p^*p_*\o_{\p(\cE)}(d)\to \o_{\p(\cE)}(d)$
is surjective. The corresponding morphism $\p(\cE)\to \p\big(p_*\o_{\p(\cE)}(d)\big)$
is an embedding over $\Delta^2$, which restricts to the $d$-uple embedding of $\p^n$ on each fiber of $p$.
By construction, the locally free sheaves $p_*\o_{\p(\cE)}(d)$ and $\cF$ are isomorphic over $\Delta^2 \setminus \{\bar 0\}$,
hence isomorphic over $\Delta^2$.
Thus, there is an embedding $j:\p(\cE)\into \p(\cF)$ over $\Delta^2$ such that
$j=i\circ \varphi$ on $\p(\cE)\setminus p^{-1}(\bar 0)$.
It follows that the closure of $i\big(X\setminus \pi^{-1}(\bar 0)\big)$ in $\p(\cE)$ is
$j(\p(\cE))$, and thus $X\cong \p(\cE)$.

Once we know that $\pi$ is a $\p^n$-bundle, we use Lemma~\ref{constructing_O(1)} to construct a
global $\o(1)$ for $\pi$.
\end{proof}

\begin{proof}[Proof of Corollary~\ref{cor}]
Let $U$ be a smooth variety, and $U\to V$ a surjective quasi-finite morphism onto an open subset $V\subset Y$.
Theorem ~\ref{flat_morphism_smooth_base} applied to the induced flat morphism $X\times_Y U\to U$
shows that $S_\pi$ has pure codimension $1$ on $V$.
\end{proof}

\begin{proof}[Proof of Theorem~\ref{X_smooth_pi_flat}]
First we prove that $S_\pi$ is either empty or has pure codimension $1$.
Given $y\in S_\pi$, let $Z\subset X$ be a complete intersection of $n$ general very ample divisors on $X$.
Then $Z$ is smooth by Bertini's Theorem, and $\pi|_Z:Z\to Y$ is finite over a neighborhood of $y$.
Our claim follows from Corollary~\ref{cor}.

If  $S_\pi= \emptyset$, then $Y$ is smooth. Indeed, given $y\in Y$, 
by Bertini's Theorem, we can take $Z$ as above such that  $\pi|_Z:Z\to Y$ 
is unramified over $y$.

Next, we show that $R_\pi$ is Zariski dense in $S_\pi$.
Let $C\subset Y$ be a smooth curve obtained as complete intersection of general very ample divisors on $Y$.
Then $X_C=\pi^{-1}(C)$ is smooth by Bertini's Theorem, and $C$ intersects every irreducible component of $S_\pi$ at general points.
Set $U:=C\setminus \overline{R_\pi}$ and $X_U:= \pi^{-1}(U)$.
The relative Picard number $\rho(X_U/U)$ equals $1$, since 
every fiber of $\pi$ over $U$ is irreducible.
Let $V\subset U$ be an open subset over which $\pi$ is a $\p^n$-bundle.
Since $\dim V=1$, every $\p^n$-bundle over $V$ is a scroll.
Hence there is a line bundle $\cL_V$  on $\pi^{-1}(V)$ such that
$\cL_V|_{X_t}\cong \o_{\p^n}(1)$ for every $t\in V$.
We can extend $\cL_V$ to a line bundle $\cL$ on $X_C$.
The restriction of $\cL$ to $X_U$ is $\pi$-ample since $\rho(X_U/U)=1$.
Since $c_1(\cL)^n\cdot X_t=1$ for every $t\in C$, all fibers of $\pi$ over $U$ are reduced.
Theorem F1 then implies that $\pi|_{X_U}:X_U\to U$ is a $\p^n$-bundle.
This shows that $\overline{R_\pi}=S_\pi$.
\end{proof}

%%%%%%%%%%%%%%%%%%%%%%%%%%%%%%%%%%%%%%%%%%%%%%%%%%%%%%%%%
%                                                       %
%                     SECTION 3                         %
%                                                       %
%%%%%%%%%%%%%%%%%%%%%%%%%%%%%%%%%%%%%%%%%%%%%%%%%%%%%%%%%

\section{Flatness criteria} \label{section:flatness}

It is often useful to have characterizations of $\p^n$-bundles without flatness assumptions.
In this context, we recall the following result of Fujita.

\begin{thmF2}[{\cite[Lemma 2.12]{fujita85}}]
Let $X$ be a smooth complex projective variety, $Y$ a normal complex projective variety,
and $\pi:X\to Y$ a surjective equidimensional  morphism.
Let $\cL$ be an ample line bundle on $X$, and
suppose that $\big(X_t,\cL|_{X_t}\big)\cong \big(\p^n,\o_{\p^n}(1)\big)$ for general $t\in Y$.
Then $Y$ is smooth, $\pi$ is a $\p^n$-bundle, and $\cL$ is a global $\o(1)$ for $\pi$.
\end{thmF2}

One reduces Theorem F2 to Theorem F1 by applying the following flatness criterion.

\begin{criterion}[{\cite[6.1.5]{EGA4}}] \label{criterion}
An equidimensional proper morphism $\pi:X\to Y$ of algebraic varieties is flat,
provided that $Y$ is smooth and $X$ is locally Cohen-Macaulay.
\end{criterion}

\begin{rem}\label{flat<->CM}
Let $\pi:X\to Y$ be a finite flat morphism of algebraic varieties.
Suppose that $Y$ is locally Cohen-Macaulay at a point $y=\pi(x)$.
We claim that $X$ is locally Cohen-Macaulay at  $x$.
Indeed, by \cite[Corollary 18.17]{eisen}, this is the case if $Y$ is
smooth at $y$. The general case can be reduced to this one by applying Noether normalization
theorem to $Y$ and observing Criterion~\ref{criterion}.
\end{rem}

More refined flatness criteria can be found in \cite{kollar_flatness} and \cite{kollar_husks}.
The problem is more delicate under the presence of everywhere nonreduced fibers.
The next example illustrates this situation.

\begin{exmp}\label{examplesing}
Let $\sigma$ be an involution of $\p^n$.
Consider the diagonal action of $\mu_2$ on $\p^n\times \A^2$, where
the action on $\p^n$ is given by $\sigma$, and
the action on $\A^2$ is given by the antipodal map.
Set $X:=(\p^n\times \A^2)/\mu_2$, $Y:=\A^2/\mu_2$, and denote by $o\in Y$ the unique singular point of $Y$.
The actions induce a proper equidimensional morphism $\pi:X\to Y$ such that $X_t\cong \p^n$ for $t\in Y\setminus\{o\}$,
while $X_o$ is not generically reduced. Moreover  $(X_o)_{red}\cong \p^n/\mu_2$.
Note that $\pi:X\to Y$ is not flat. This can be seen by considering the induced morphism
$X\times_Y\A^2 \to \A^2$ and applying Theorem~\ref{flat_morphism_smooth_base}.
\end{exmp}

We end the paper with some flatness criteria.

\begin{prop}\label{prop:flatness}
Let $\pi:X\rightarrow Y$ be a projective surjective equidimensional morphism between normal algebraic varieties.
Then $\pi$ is flat at $x\in X$, provided that one of the following conditions holds.
\begin{enumerate}
	\item  $\pi$ has connected fibers, $X$ is smooth at $x$, and there is a  surjective quasi-finite morphism from a smooth variety
		onto a neighborhood of $\pi(x)$ in $Y$ that is  \'etale in codimension $1$.
	\item $X$ is locally Cohen-Macaulay at $x$, and there is a finite flat morphism from a smooth variety
		onto a neighborhood of $\pi(x)$ in $Y$.
\end{enumerate}
\end{prop}

\begin{rem}
The last part of condition (1) above holds for example
if the fiber through $x$ has a generically reduced irreducible component, or if $\pi(x)$ is a quotient singularity of $Y$.
\end{rem}

\begin{proof}
Assume (1) holds.
Let $U\to V$ be a  surjective quasi-finite morphism from a smooth variety
onto a neighborhood of $\pi(x)$, \'etale in codimension $1$.
Set $X':=X\times_Y U$, and note that it is irreducible since $\pi$ has connected fibers and $X$ is normal (and so is the general fiber of $\pi$).
The induced morphism $\varphi:X'\to X$ is quasi-finite and \'etale in codimension $1$.
By purity of the branch locus (see \cite{AK}), we conclude that $\varphi$ is \'etale, and thus $X'$ is smooth at any point $x'$ over $x\in X$.
The induced morphism $\pi':X'\to U$ is then flat at $x'$  by Criterion~\ref{criterion}, hence $\pi$ is flat at $x$
(see \cite[Theorem III.9.9]{hartshorne}).

Assume (2) holds.
Let $U\to V$ be a finite flat morphism from a smooth variety onto a neighborhood of $\pi(x)$.
Set $X':=X\times_Y U$.
By Remark~\ref{flat<->CM}, $X'$ is locally Cohen-Macaulay at any point $x'$ over $x\in X$.
The induced morphism $\pi':X'\to U$ is then flat at $x'$ by  Criterion~\ref{criterion},  hence $\pi$ is flat at $x$.
\end{proof}

\begin{rem}
Theorem F2 can be generalized to normal varieties $X$ (see \cite[Proposition 4.10]{araujo_druel}).
Other characterizations of scrolls in a similar vein appear in \cite{hoering}.
\end{rem}

\bibliographystyle{amsalpha}
\bibliography{calorina}

\providecommand{\bysame}{\leavevmode\hbox to3em{\hrulefill}\thinspace}
\providecommand{\MR}{\relax\ifhmode\unskip\space\fi MR }
\providecommand{\MRhref}[2]{%
  \href{http://www.ams.org/mathscinet-getitem?mr=#1}{#2}
}
\providecommand{\href}[2]{#2}
\begin{thebibliography}{OSS80}

\bibitem[AD12]{araujo_druel}
C.~Araujo and S.~Druel, \emph{On codimension 1 del pezzo foliations on
  varieties with mild singularities}, pre-print math.AG/1210.4013, 2012.

\bibitem[AK71]{AK}
A.~Altman and S.~Kleiman, \emph{On the purity of the branch locus}, Comp. Math.
  \textbf{23} (1971), no.~4, 461--465.

\bibitem[AM97]{andreatta_mella97}
M.~Andreatta and M.~Mella, \emph{Contractions on a manifold polarized by an
  ample vector bundle}, Trans. Amer. Math. Soc. \textbf{349} (1997), no.~11,
  4669--4683.

\bibitem[Art82]{artin}
M.~Artin, \emph{{B}rauer-{S}everi varieties}, Brauer groups in ring theory and
  algebraic geometry (Wilrijk, 1981), Lecture Notes in Math., vol. 917,
  Springer, Berlin, 1982, pp.~194--210.

\bibitem[BT82]{bott_tu}
R.~Bott and L.~Tu, \emph{Differential forms in algebraic topology}, Graduate
  Texts in Mathematics, vol.~82, Springer, New York, 1982.

\bibitem[Eis95]{eisen}
D.~Eisenbud, \emph{Commutative algebra with a view toward algebraic geometry},
  GTM, Springer-Verlag, New York, 1995.

\bibitem[Fuj75]{fujita75}
T.~Fujita, \emph{On the structure of polarized varieties with {$\Delta
  $}-genera zero}, J. Fac. Sci. Univ. Tokyo Sect. IA Math. \textbf{22} (1975),
  103--115.

\bibitem[Fuj87]{fujita85}
\bysame, \emph{On polarized manifolds whose adjoint bundles are not
  semipositive}, Algebraic geometry, Sendai, 1985, Adv. Stud. Pure Math.,
  vol.~10, North-Holland, Amsterdam, 1987, pp.~167--178.

\bibitem[Gro66]{EGA4}
A.~Grothendieck, \emph{\'{E}l\'ements de g\'eom\'etrie alg\'ebrique. {IV}.
  \'{E}tude locale des sch\'emas et des morphismes de sch\'emas (3{\`e}me
  partie)}, Inst. Hautes \'Etudes Sci. Publ. Math. (1966), no.~28.

\bibitem[Har77]{hartshorne}
R.~Hartshorne, \emph{Algebraic geometry}, Springer-Verlag, New York, 1977,
  Graduate Texts in Mathematics, No. 52.

\bibitem[HN12]{hoering}
Andreas H\"oring and Carla Novelli, \emph{Mori contractions of maximal length},
  Preprint {\tt arXiv:1201.4009}, 2012.

\bibitem[Kol95]{kollar_flatness}
J.~Koll{\'a}r, \emph{Flatness criteria}, J. Algebra \textbf{175} (1995), no.~2,
  715--727.

\bibitem[Kol11]{kollar_husks}
J{\'a}nos Koll{\'a}r, \emph{Simultaneous normalizations and algebra husks},
  Asian J. Math. \textbf{15} (2011), no.~3, 437--450.

\bibitem[Mil80]{milne}
J.S. Milne, \emph{{\'E}tale {C}ohomology}, Universitext, Princeton University
  Press, Princeton, 1980.

\bibitem[OSS80]{OSS_vector_bdles_on_Pn}
C.~Okonek, M.~Schneider, and H.~Spindler, \emph{Vector bundles on complex
  projective spaces}, Progress in Mathematics, vol.~3, Birkh\"auser Boston,
  Mass., 1980.

\bibitem[Ser66]{SerreNorm}
J.~P. Serre, \emph{Prolongement de faisceaux analytiques coh\'erents}, Ann.
  Inst. Fourier \textbf{16} (1966), no.~1, 363--374.

\bibitem[Siu92]{siu_rigidity}
Y.~T. Siu, \emph{Nondeformability of the complex projective space}, J. Reine
  Angew. Math. \textbf{399, Errata 431} (1992), 208--219, Errata 65--74.

\bibitem[Uen75]{ueno_LNM439}
K.~Ueno, \emph{Classification theory of algebraic varieties and compact complex
  spaces}, Springer-Verlag, Berlin, 1975, Lecture Notes in Mathematics, Vol.
  439.

\bibitem[Voi02]{voisin1}
C.~Voisin, \emph{Hodge theory and complex algebraic geometry {I}}, Cambridge
  studies in advanced mathematics, vol.~76, CUP, Cambridge, UK, 2002.

\end{thebibliography}

\end{document}